\documentclass{amsart}
\usepackage{graphicx}
\usepackage{newlattice}

\DeclareMathOperator{\length}{length}

\newcommand{\oa}{\ol \bga}

\theoremstyle{plain}
\newtheorem{theorem}{Theorem}

\newtheorem{lemma}[theorem]{Lemma}

\theoremstyle{definition}

\begin{document}
\title[Congruences of the fork extensions. I. CEP]
{Congruences of the fork extensions. I.\\
The Congruence Extension Property}  
\author{G. Gr\"{a}tzer} 
\address{Department of Mathematics\\
  University of Manitoba\\
  Winnipeg, MB R3T 2N2\\
  Canada}
\email[G. Gr\"atzer]{gratzer@me.com}
\urladdr[G. Gr\"atzer]{http://server.maths.umanitoba.ca/homepages/gratzer/}

\date{July 2, 2013}
\subjclass[2010]{Primary: 06C10. Secondary: 06B10.}
\keywords{semimodular lattice, fork extension, congruence.}

\begin{abstract}
For a slim, planar, semimodular lattice, 
G.~Cz\'edli and E.\,T.~Schmidt introduced the fork extension in 2012.
In this note we prove that 
the fork extension has the Congruence Extension Property.
\end{abstract}

\maketitle

\section{Introduction}\label{S:Introduction}

\begin{figure}[b]
\centerline{\includegraphics{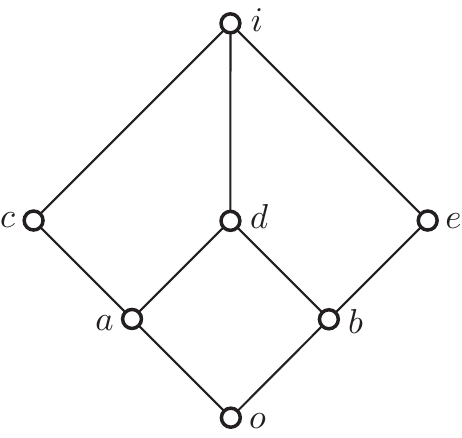}}
\caption{The lattice $\SfS 7$}\label{F:s7}
\end{figure}

Let $L$ and $K$ be lattices and let $K$ be an extension of $L$.
We say that the \emph{Congruence Extension Property} holds, 
if every congruence of $L$ has an extension to~$K$.

Let $L$ be a slim, planar, semimodular lattice, an \emph{SPS lattice}.
As in G.~Cz\'edli and E.\,T.~Schmidt~\cite{CSb}, 
\emph{inserting a fork} to $L$ at the covering square~$S$, 
firstly, replaces~$S$ by a~copy of $\SfS 7$ 
(see the lattice $\SfS 7$ in Figure~\ref{F:s7}). 

Secondly, if there is a chain 
$u\prec v\prec w$ such that~the element $v$ has just been added 
and 
$
   T = \set{x=u \mm z, z, u, w=z \jj u}
$
is a covering square in the lattice~$L$ 
(and so $u \prec v \prec w$ is not on the boundary of $L$) 
but $x \prec z$ at the present stage of the construction,
then we insert a new element~$y$ 
such that $x \prec y \prec z$ and $y \prec v$.

Let $L[S]$ denote the lattice we obtain when the procedure terminates. 
We say that $L[S]$ is obtained from
$L$ by \emph{inserting a fork to $L$} at the covering square~$S$.
See Figure~\ref{F:forks} for an illustration.

\begin{figure}[hbt]
\centerline{\includegraphics{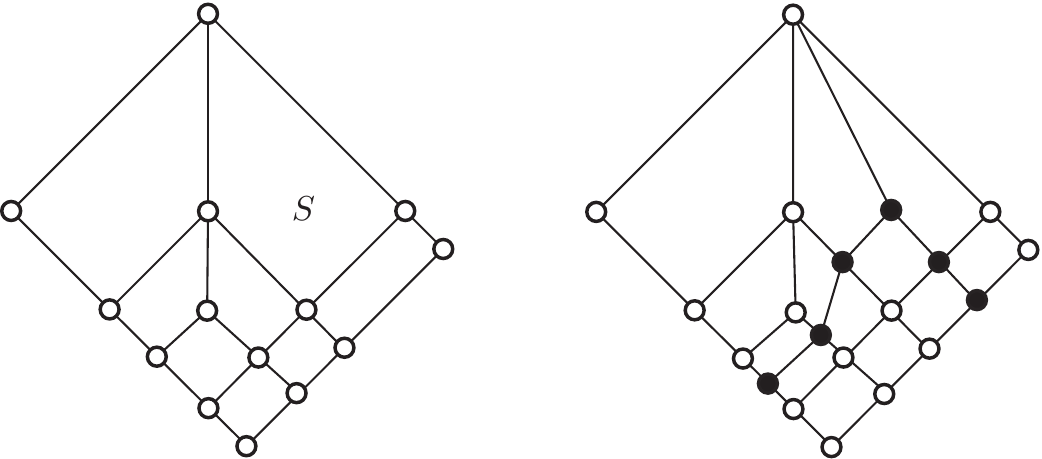}}
\caption{Inserting a fork at $S$}\label{F:forks}
\end{figure}

Let $L$ be an SPS lattice
and $S$ a covering square of $L$.
In~this note we will examine the connections between the congruence lattice of $L$ and the congruence lattice of $L[S]$.

\begin{theorem}\label{T:extension}
Let $L$ be a slim, semimodular, planar lattice. 
Let $S$ be a covering square of $L$.
Then the extension of $L$ to $L[S]$ 
has the Congruence Extension Property.
\end{theorem}

We will use the notations and concepts of lattice theory 
as in \cite{LTF}.

\section{Congruences of finite lattices}\label{S:Congruences}

The following lemma is trivial but very useful.

\begin{lemma}\label{L:technical}
Let $L$ be a finite lattice. 
Let $\bgd$ be an equivalence relation on $L$
with intervals as equivalence classes.
Then $\bgd$ is a congruence relation if{}f the following condition 
and its dual hold:
\begin{equation}\label{E:cover}
\text{If $x$ is covered by $y \neq z$ in $L$ 
and $x \equiv y\,(\tup{mod}\, \bgd)$,
then $z \equiv y \jj z\,(\tup{mod}\,\bgd)$.}\tag{C${}_{\jj}$}
\end{equation}
\end{lemma}

\begin{proof}
We want to prove that if $x \leq y$ and $\cng x=y (\bgd)$,
then $\cng x \jj z = y \jj z(\bgd)$.
The proof is a trivial induction first on $\length[x, y]$
and then on $\length[x, x \jj z]$.
\end{proof}

Let (C${}_{\mm}$) denote the dual of (C${}_{\jj}$).

\section{The fork construction}\label{S:forks}

Let $L$ be an SSP lattice. 
Let $S$ be a covering square of $L$.

We need some notation for the $L[S]$ construction, 
see Figure~\ref{F:forksdetails}. 

\begin{figure}[h]
\centerline{\includegraphics{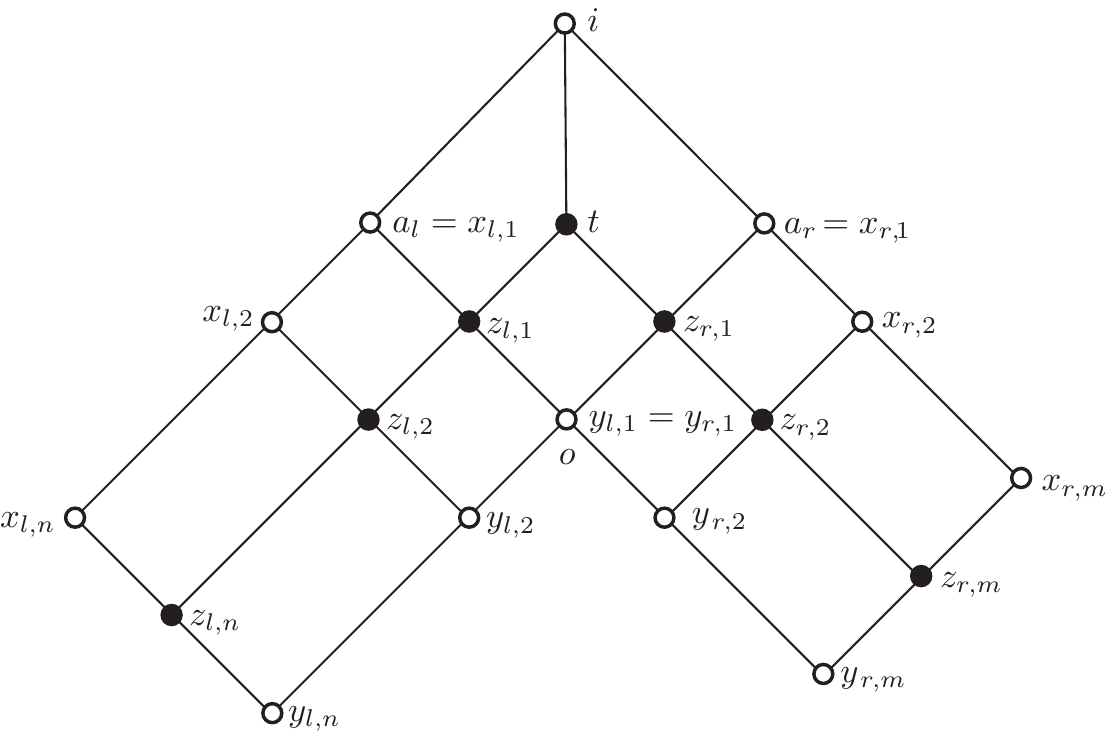}}
\caption{Notation for the fork construction}\label{F:forksdetails}
\end{figure}

We start the construction 
by adding the elements $t$, $z_{l,1}$, and $z_{r,1}$
so that the set $\set{o, z_{l,1}, z_{r,1}, a_l, a_r, t,i}$ 
forms a sublattice $\SfS 7$.

Let $a_l = x_{l,1}$, $o = y_{l,1}$. 
If $k$ is the largest number so that 
$x_{l,k}$, $y_{l,k}$, and $z_{l,k}$ have already been defined, and
\[
   T = \set{y_{l,k+1} = x_{l,k+1} \mm y_{l,k}, x_{l,k+1}, y_{l,k},
    x_{l,k} = x_{l,k+1} \jj y_{l,k}}
\]
is a covering square in $L$, then we add the element $z_{l,k+1}$,
so we get two new covering squares 
$\set{y_{l,k+1}, x_{l,k+1}, y_{l,k}, x_{l,k}}$
and $\set{z_{l,k+1}, y_{l,k+1}, z_{l,k}, y_{l,k}}$.
We proceed similarly on the right.

So $L[S]$ is constructed by inserting the elements in the set
\[
   F[S] = \set{t,z_{l,1} \succ \dots \succ z_{l,n}, 
          z_{r,1} \succ \dots \succ z_{r,m}}
\]
so that $\set{o, z_{l,1}, z_{r,1}, a_l, a_r, t,i}$ 
is a sublattice $\SfS 7$,
moreover, $x_{l,i} \succ z_{l,i} \succ y_{l,i}$ for $i = 1, \dots, n$, and $x_{r,i} \succ z_{r,i} \succ y_{r,i}$ for $i = 1, \dots, m$.
The new elements are black filled in Figure~\ref{F:forksdetails}.  

\begin{lemma}\label{L:easy}
Let $L$ be an SSP lattice with the covering square~$S$.
Then $L$ is a sublattice of $L[S]$. 
Therefore, every element $x$ of $L[S]$ 
has an upper cover $\ol x$ and a lower cover $\ul x$ in $L$.
Specifically, $\ol z_{l,i} = x_{l,i}$ for $i = 1, \dots, n$,
$\ol z_{r,i} = x_{r,i}$ for $i = 1, \dots, m$, and $\ol t = i$.

If $x, y \in L[S]$ and $x \nin F[S]$, then 
\[
   x \jj y = x \jj \ol y,
\]
and dually.
\end{lemma}

The following results are well known.

\begin{lemma}\label{L:known}
Let $L$ be an SSP lattice. 
\begin{enumeratei}
\item An element of $L$ has at most two covers.

\item Let $a \in L$. 
Let $a$ cover the three elements $x_1$, $x_2$, and $x_3$.
Then the set $\set{x_1,x_2,x_3}$ generates an $\SfS 7$ sublattice.

\item If the elements $x_1$, $x_2$, and $x_3$ are adjacent, 
then the $\SfS 7$ sublattice of \tup{(ii)} is a cover-preserving sublattice.

\end{enumeratei}
\end{lemma}

\section{Proving Theorem~\ref{T:extension}}\label{S:fork}

We prove Theorem~\ref{T:extension} (and somewhat more) 
in the next three lemmas.

\begin{lemma}\label{L:extend1}
Let $L$ be an SSP lattice. 
Let $S = \set{0,a_l,a_r,i}$ be a covering square of $L$.  
Let $\cng o=i(\bga)$.
Then $\bga$ extends to $L[S]$ uniquely to a congruence~$\oa$ of~$L[S]$.
\end{lemma}

\begin{proof}
We define $\oa$ as the partition:
\[
  \gp =  \setm{[u,v]_{L[S]}}{[u,v]_L \text{\ is a congruence class of $\bga$}}
\]
To verify that $\gp$ is indeed a partition of $L[S]$, let
\[
   A =  \UUm{[u,v]_{L[S]}}{[u,v]_L 
        \text{\ is a congruence class of $\bga$}}.
\]
Clearly, $L \ci A$. Since $\cng o=i(\bga)$, there is a 
congruence class $[u,v]_{L}$ containing $o$ and $i$. 
Hence  $[u,v]_{L[S]} \in \gp$. Therefore, $t, z_{l,1}, z_{r,1} \in A$.
For $i = 1, \dots, n$,  there is a 
congruence class $[u_i,v_i]_{L}$ containing $x_{l,i}$ and $ y_{l,i}$.
Hence $z_{l,i} \in [u_i,v_i]_{L[S]} \ci A$ and symmetrically  $z_{l,i} \in A$
for $i = 1, \dots, m$. This proves that $A = L$.

Next we show that the sets in $\gp$ are pairwise disjoint.

Finally, we verify the substitution properties. 
Let $a,b,c \in L[S]$ and $\cng a=b(\gp)$.
Then there exist $u,v \in L$ with  $\cng u=v(\bga)$ such that  $a,b \in [u,v]_{L[S]}$. 
There is also an interval $[u',v']_{L[S]} \in \gp$ with $c \in [u',v']_{L[S]}$.
Since $\cng u=v(\bga)$ and  $\cng u'=v'(\bga)$, it follows that 
$\cng u\jj u'=v \jj v'(\bga)$, so there is a congruence class $[u'',v'']$ of $\bga$ 
containing $u\jj u'$ and $v \jj v'$. 
So $u\jj u', v \jj v' \in [u'',v'']_{L[S]}$, 
verifying the substitution property for joins. 
The dual proof verifies the substitution property for meets.

The uniqueness statement is obvious.
\end{proof}

\begin{lemma}\label{L:extend2}
Let $L$ be an SSP lattice. 
Let $S = \set{0,a_l,a_r,i}$ be a covering square of~$L$. 
Let $\bga$ be a congruence of $L$ satisfying  $\bga \restr S = \zero_S$.
Then $\bga$ extends to  $L[S]$.
\end{lemma}

\begin{proof}
We are going to define the minimal congruence $\oa$ 
of $L[S]$ extending $\bga$.

For $i = 1, \dots, n$, define $\ul i$ 
as the smallest element in $\set{1, \dots, n}$
satisfying $\cng x_{l,i} = x_{l,\ul i}(\bga)$; 
let $\ol i$ be the largest one.
 
We define $\oa$ as the partition:
\begin{align}\label{E:zero}
  \gp_d =  &\setm{[u,v]_{L[S]}}{[u,v]_L 
  \text{\ is a congruence class of $\bga$}}\\
 & \uu \set{t}
  \uu \setm{[z_{l,\ul i}, z_{l,\ol i}]}{i=1, \dots, n}
   \uu \setm{[z_{r,\ul i}, z_{r,\ol i}]}{i=1, \dots, m}.\notag
\end{align}
Clearly, $\oa$ is a partition. 
To see that $\oa$ is a congruence, we use Lemma~\ref{L:technical}.
To verify (C${}_\jj$), let $x \prec y,z \in L[S]$, $y \neq z$, 
and $\cng x = y (\oa)$. We want to prove that $\cng z = z \jj y (\oa)$.
By \eqref{E:zero}, either $\cng x=y(\bga)$ 
or $x,y \in [z_{l,\ul i}, z_{l,\ol i}]$ for some $i=1, \dots, n$
(or symmetrically).

First, let $\cng x=y(\bga)$. If $z \in L$, then $\cng z = z \jj y (\bga)$
since $\bga$ is a congruence, so by \eqref{E:zero}, 
$\cng z = z \jj y (\oa)$. If $z \nin L$, that is, if $z \in F[S]$,
then $t \neq i$, since $i$ does not cover an element not in $F[S]$.
So $z = z_{l, i}$ for some $i=1, \dots, n$ (or symmetrically).
Since $x \prec z = z_{l, i}$ and $x \in L$, 
we get that $x = z_{l, i}$, $i > 1$ and $x = y_{l, i-1}$. 
It easily follows that $\cng z=y\jj z(\bga)$.

Second, let $x,y \in [z_{l,\ul i}, z_{l,\ol i}]$ for some $i=1, \dots, n$
(or symmetrically). Then $x = z_{l, j}$ and $y = z_{l, j-1}$
with $\cng z_{l, j}=z_{l, j-1}(\bga)$ by \eqref{E:zero}, 
and so $\cng z = z \jj y (\oa)$.
\end{proof}

\begin{lemma}\label{L:extend3}
Let $L$ be an SSP lattice. 
Let $S = \set{0,a_l,a_r,i}$ be a covering square of $L$. 
Let $\bga$ be a congruence of $L$ satisfying  
$\cng a_l=i(\bga)$ and $a_r \not\equiv i\,(\tup{mod}\, \bga)$.
Then $\bga$ extends to a unique congruence $\oa$ of $L[S]$.
\end{lemma}

\begin{proof}

\begin{figure}[tbh]
  \centerline{\includegraphics{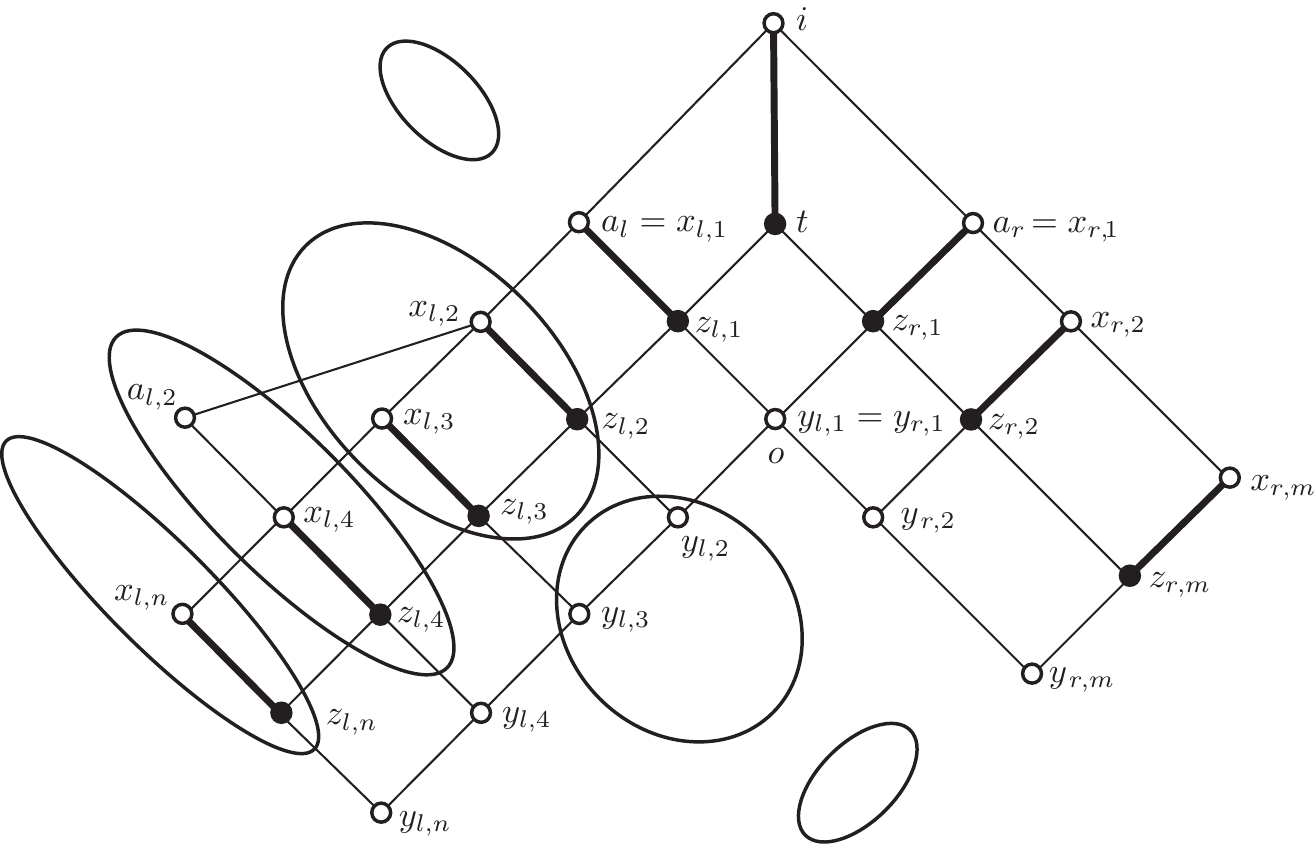}}
  \caption{Illustrating $\oa$ for Lemma \ref{L:extend3}}\label{F:alphabar}
\end{figure}

Define 
\begin{align}
   i/\oa &= i/\bga \uu \setm{z_{l,i}}
               {x_{l,i} \equiv x_{l,1}\,(\tup{mod}\,\bga)}
                &&\text{\q for $i = 1, \dots, n$};\label{E:i}\\
   x_{l,i}/\oa &=  x_{l,i}/\bga \uu \setm{z_{l,k}}
               {x_{l,i} \equiv x_{l,k}\,(\tup{mod}\,\bga)}
                      &&\text{\q for $i,k = 2, \dots, n$};\label{E:ii}\\
  x_{r,j}/\oa &= x_{r,j}/\bga \uu \setm{z_{r,k}}
                  {x_{r,j} \equiv x_{r,k}\,(\tup{mod}\,\bga)}
               &&\text{\q for $j,k = 2, \dots, m$}\label{E:iii}.
\end{align}
We define $\oa$ as follows: 

Let $x/\oa$ be as defined above for 
\[
   x \in i/\oa \uu \UUm{x_{l,i}/\oa}{i = 2, \dots, n} 
      \uu \UUm{x_{r,j}/\oa}{j = 2, \dots, m};
\]
let $x/\oa = x/\bga$ otherwise.

We have to prove that $\oa$ is a congruence. 

The sets $x/\oa$, $x \in L[S]$, 
are pairwise disjoint intervals, 
so they define an equivalence relation $\oa$ 
with intervals as equivalence classes.

We again use Lemma~\ref{L:technical}.
To verify (C${}_\jj$), let $x \prec y,z \in L[S]$, $y \neq z$, 
and $\cng x = y (\oa)$. We want to prove that $\cng z = z \jj y (\oa)$
There are four cases to consider.

Case 1: $x, y \in L$. Then $\cng x=y(\bga)$.
If $z \in L$, then $\cng x \jj z =y\jj z(\bga)$, since $\bga$ is a congruence of $L$, 
so $\cng x \jj z =y\jj z(\oa)$. So we can assume that $z \in F[S]$. 

Since $x \prec y$ in $L$, it follows that $x, y \in \id o$ 
or $x, y \in \fil {x_{l,n}} \uu \fil {x_{r,m}}$.

\q Case 1.1: $x, y \in \id o$ and $z \in F[S]$. 
Let $z = z_{l,i}$ ($z = z_{r,i}$ proceeds similarly). 
Then $z \jj y = z_{l,i-1}$.
We have to verify that $\cng z_{l,i} = z_{l,i-1} (\oa)$.

Note that $x \leq x_{l,i}$ and $x_{l,i} \jj y = x_{l,i-1}$.
Since $\cng x=y(\bga)$, 
it follows that $\cng x_{l,i}=x_{l,i-1}(\bga)$.
By \eqref{E:ii}, $\cng z_{l,i}=z_{l,i-1}(\oa)$, as claimed.

\q Case 1.2: $x, y \in \fil {x_{l,n}} \uu \fil {x_{r,m}}$ 
and $z \in F[S]$.
This cannot happen in view of $x \prec z$.

Case 2: $x \in L$, $y \in F[S]$. 

\q Case 2.1: $y = z_{l,i}$. 
This cannot happen because by \eqref{E:ii} there is no
$x \in L$, $x < z_{l,i}$ in $x_{l,i}/\oa$.

\q Case 2.2: $y = z_{r,i}$. Then $x = y_{r,i}$.

\qq Case 2.2.1: $i \neq 1$. 
Then $z = y_{r, i-1}$ and $\cng {z_{r,i-1}}={y\jj z = z_{r,i-1}}(\oa)$
by~\eqref{E:iii}.

\qq Case 2.2.2: $i = 1$, so $y = y_{r,1}$. 
Then $z = z_{l, 1}$ and we proceed as in Case~2.2.1.

\q Case 2.3: $y = t$.
This cannot happen because there is no
$x \in L$ with $x \prec t$.

Case 3:  $x \in F[S]$, $y \in L$. 
Then $x = z_{l, i}$ and $y = x_{l, i}$. Joining them with $z$, if we get two distinct elements, they are of the same form, of the pair $t$, $i$, all of which are congruence modulo $\oa$. 

Case 4: $x, y \in F[S]$. Let $x = z_{l,i}$ and $y = z_{l,i-1}$ 
(or symmetrically). 
Then $z = x_{l,i}$ and the conclusion follows from \eqref{E:iii}.

We have verified condition (C${}_{\jj}$) of Lemma~\ref{L:technical}.
The final step is to verify (C${}_{\mm}$).

So let $x$ cover $y \neq z$ in $L[S]$ 
and $\cng x = y(\oa)$. 
We want to prove that $\cng z = x \mm z(\oa)$.
We distinguish the same four cases as for (C${}_{\mm}$).

Case 1: $x, y \in L$. Then $\cng x=y(\bga)$.
If $z \in L$, then $\cng x \mm z =y\mm z(\bga)$, 
so $\cng x \mm z =y\mm z(\oa)$. So we can assume that $z \in F[S]$. 

Since $z \in F[S]$ is covered by exactly one element $\ol z$ of $L$,
so $x = \ol x$, and $y \mm z, z \in x/\oa$, 
and the claim follows.

Case 2: $x \in L$, $y \in F[S]$, say $y = z_{l,i}$. 
Then $x = x_{l,i}$, so $y = z_{l,i}$ and $z = x_{l,i+1}$.
By \eqref{E:ii}, $\cng y =y \mm z(\oa)$, as claimed.

Case 3:  $x \in F[S]$, $y \in L$. 
This cannot happen because there is no
$x \in F[S]$ and $y \in L$ with $x \succ y$ and $\cng x=y(\oa)$.

Case 4: $x, y \in F[S]$. Let $x = z_{l,i}$ and $y = z_{l,i+1}$ 
(or symmetrically). 
Then $z = y_{l,i}$ and the conclusion follows from the fact that
$\cng y_{l,i} = y_{l,i+1}(\bga)$.

The uniqueness statement is obvious.
\end{proof}

\end{document}